\documentclass[11pt]{article}
\usepackage{amsmath,amsthm,amsfonts,amssymb,amscd, amsxtra, mathrsfs}
\usepackage{url}
\usepackage[margin=2.5cm,nohead]{geometry}
\usepackage{color}
\newcommand{\banacha}{X}
\newcommand{\banachb}{Y}

\newcommand{\banachh}{H}
\newtheorem{theorem}{Theorem}
\newtheorem{lemma}[theorem]{Lemma}
\newtheorem{definition}{Definition}
\newtheorem{corollary}[theorem]{Corollary}
\newtheorem{proposition}[theorem]{Proposition}

\newtheorem{remark}{Remark}
\begin{document}

\title{Local  convergence  of Newton's  method for solving  generalized equations with monotone operator}

\author{Gilson N. Silva \thanks{CCET/UFOB,  CEP 47808-021 - Barreiras, BA, Brazil (Email: {\tt  gilson.silva@ufob.edu.br}).}
}
\date{July 29, 2016}

\maketitle
\maketitle
\begin{abstract}
In this paper we study Newton's method for solving the generalized equation $F(x)+T(x)\ni 0$ in Hilbert spaces, where $F$ is a Fr\'echet differentiable function and $T$ is set-valued and maximal monotone. We show that this method is local  quadratically convergent to a solution. Using the idea of majorant condition on the nonlinear function which is associated to the generalized equation, the convergence of the method,  the optimal convergence radius and results on the convergence rate are established. The advantage of working with a majorant condition rests in the fact that it allow to unify several convergence results pertaining to  Newton's method.
\end{abstract}
\medskip


\medskip
\maketitle

\section{Introduction}\label{sec:int}
The idea of solving a generalized equation of the form
\begin{equation} \label{eq:ipi}
\mbox {Find} ~ x ~\mbox{such that} ~  F(x) +T(x) \ni 0,
\end{equation}
where $F:{\Omega}\to \banachh$ is a Fr\'echet differentiable function, $\banachh$ is a Hibert space, $\Omega\subseteq \banachh$ an open set and $T:\banachh \rightrightarrows  \banachh$ is a set-valued and maximal monotone,  plays a huge role in classical analysis and its applications. For instance, systems of nonlinear equations and abstract inequality systems. If $\psi: \banachh \to (-\infty, +\infty]$ is a proper lower semicontinuous convex function and
$$
T(x)=\partial \psi(x) =\{u\in \banachh ~:~ \psi(y) \geq \psi(x) + \langle u,y-x\rangle\}, \qquad \forall\quad y\in \banachh,
$$
then \eqref{eq:ipi} becomes the variational inequality problem
$$
F(x) +\partial \psi(x) \ni 0,
$$
including linear and nonlinear complementary problems; additional comments about such problems can be found in \cite{Dontchev1996, DontchevRockafellar2009, DontchevRockafellar2010, FerreiraSilva, josephy1979, Robinson1972_2,  Uko1996, Wang2015}.

Newton's method has  been extended in order  to solve nonlinear systems of equalities and inequalities (see \cite{Daniel1973} ). In particular,  Robinson in \cite{Robinson1972_2}  generalized Newton's method for solving problems of the form
$$
F(x)\in C,
$$
which becomes  the usual Newton's method  to the special case in which $C$ is the degenerate cone $\{0\}\subset Y$.

A Newton method for solving \eqref{eq:ipi} utilizes the iteration
\begin{equation} \label{eq:ipi1}
  F(x_k) + F'(x_k)(x_{k+1}-x_k)+ F(x_{k+1}) \ni 0, \qquad ~k=0,1,...
\end{equation}
for $x_0$ a given initial point. As is well known, the generalized equation \eqref{eq:ipi} covers huge territory in classical analysis and its applications. When $F\equiv 0$, the iteration \eqref{eq:ipi1} becomes the standard Newton method for solving the nonlinear equation $F(x)=0,$
\begin{equation} \label{eq:ipi2}
  F(x_k) + F'(x_k)(x_{k+1}-x_k)= 0, \qquad  k=0,1,....
\end{equation}

In \cite{chang2015, UkoArgyros2009} under a majorant condition and generalized Lipschitz condition, local and semi local convergence, quadratic rate and estimate of the best possible convergence radius of Newton's method as well as uniqueness of the solution for solving generalized equation were established.

It is well-known that an assumption used to obtain quadratic convergence of Newton's method \eqref{eq:ipi1},   for solving equation \eqref{eq:ipi},  is  the Lipschitz continuity of  $F'$  in a neighborhood of the solution.  Indeed,   keeping control of the derivative is an important point in the convergence analysis of Newton's method.  On the other hand, a couple of papers have dealt with the issue of convergence analysis of the Newton's method,  for solving  the equation $F(x)=0$,   by relaxing the assumption of Lipschitz continuity of $F'$, see for example \cite{FerreiraSvaiter2009, silva2016, Wang1999, Zabrejko1987}. The advantage of working with a majorant condition rests in the fact that it allow to unify several convergence results pertaining to  Newton's method; see  \cite{FerreiraSvaiter2009, Wang1999}. In this paper we work with the majorant  condition introduced in \cite{FerreiraSvaiter2009}. The analysis presented provides a clear relationship between the majorant function and the function defining the generalized equation. Also, it allows us to obtain the optimal convergence radius for  the method with respect to the majorant condition and uniqueness of solution.  The analysis of this method, under Lipschitz's condition and Smale's condition,  are provided  as special case.

The organization of the paper is as follows. In Section~\ref{sec:int.1},  some notations and important results  used  throughout  the paper are presented. In Section \ref{lkant}, the main result is stated and  in   Section~\ref{sec:PMF} properties of the majorant function,  the main  relationships   between the majorant function and the nonlinear operator, the uniqueness of the solution and the optimal convergence radius are established. In Section~\ref{sec:proof}, the main result is proved and in the last section some applications of this result are given.
\section{Preliminaries} \label{sec:int.1}
The following notations and results are used throughout our presentation. Let $\banachh$ be a Hilbert space with scalar product $\langle ., .\rangle$  and norm $\|.\|$, the {\it open} and {\it closed balls} at $x$ with radius $\delta\geq 0$ are denoted, respectively, by $ B(x,\delta)$ and  $B[x,\delta]$.

We denote by ${\mathscr L}(\banacha,\banachb)$ the {\it space consisting of all continuous linear mappings} $A:\banacha \to \banachb$ and   the {\it operator norm}  of $A$ is defined  by $  \|A\|:=\sup \; \{ \|A x\|~:  \|x\| \leqslant 1 \}.$ A bounded linear operator $G:\banachh \to \banachh$ is called a positive operator if $G$ is a self-conjugate and $\langle Gx,x\rangle \geq 0$ for each $x\in \banachh$. The {\it domain} and the {\it range} of $G$ are, respectively,  the sets $ \mbox{dom}~G:=\{x\in \banachh ~: ~ G(x)\neq \varnothing\} $ and $ \mbox{rge}~G:=\{y\in \banachh ~: ~ y \in G(x) ~\emph{for some} ~x\in \banacha\}$. The {\it inverse} of $G$  is the set-valued mapping  $G^{-1}:\banachh \rightrightarrows  \banachh$ defined by $ G^{-1}(y):=\{x \in \banachh ~: ~ y \in G(x)\}$.

Now, we recall notions of monotonicity for set-valued operators.
\begin{definition}\label{def.mono}
Let $T:\banachh \rightrightarrows  \banachh$ be a set-valued operator. $T$ is said to be monotone if for any $x,y\in \mbox{dom}~{T}$ and, $u \in T(y)$, $v\in T(x)$ implies that the following inequality holds:
$$
\langle u-v,y-x \rangle \geq 0 .
$$
\end{definition}
A subset of $\banachh \times \banachh$ is monotone if it is the graph of a monotone operator. If $ \varphi: \banachh \to (-\infty, +\infty]$ is a proper function then the subgradient of $\varphi$ is  monotone.
\begin{definition}
Let $T:\banachh \rightrightarrows  \banachh$ be monotone. Then $T$ is maximal monotone if the following implication holds for all $x,u\in \banachh$:
\begin{equation}
\langle u-v,y-x \rangle \geq 0 \quad \mbox{for each} \quad y\in \emph{dom}{T} \quad \mbox{and}\quad v\in T(y) \Rightarrow \quad x\in \emph{dom}{T} \quad \emph{and} \quad v\in T(x).
\end{equation}
\end{definition}
An example of maximal monotone operator is the subdifferential of a proper, lower semicontinuous, convex function $ \varphi: \banachh \to (-\infty, +\infty]$.
The following result can de found in \cite{Wang2015}.
\begin{lemma}\label{eq:plm}
Let $G$ be a positive operator. The following statements about $G$ hold:
\begin{enumerate}
\item $\|G^2\|=\|G\|^2$;
\item If $G^{-1}$ exists,  then $G^{-1}$ is a positive operator.
\end{enumerate}
\end{lemma}
As a consequence of this result we have the following result:
\begin{lemma}\label{eq:adjunt}
Let $G$ be a positive operator. Suppose that $G^{-1}$ exists, then for each $x\in \banachh$ we have
$$
\langle Gx,x\rangle \geq \frac{\|x\|^2}{\|G^{-1}\|}.
$$
\end{lemma}
\begin{proof}
See Lemma~2.2 of \cite{Uko1996}.
\end{proof}

Let $G:\banachh \to \banachh$ be a bounded linear operator. We will use the convention that $\widehat{G}:=\frac{1}{2}(G+G^*)$ where $G^*$ is the conjugate operator of $G$. As we can see, $\widehat{G}$ is a self-conjugate operator.
From now, we assume that $T:\banachh \rightrightarrows  \banachh$ is a set-valued  maximal monotone operator and $F: \banachh \to \banachh$ is a Fr\'echet derivative function. The next result is of major importance to prove the good definition of Newton's method. Its proof can be found in \cite[Lemma~1, p.189]{Smale1986}.
\begin{lemma}[Banach's lemma]\label{eq.banachlemma}
Let $B: \banachh \to \banachh$ be a bounded linear operator and $I:\banachh \to \banachh$ the identity operator. If $\|B-I\|<1$ then $B$ is invertible and $\|B^{-1}\|\leq 1/(1-\|B-I\|)$.
\end{lemma}
\section{Local analysis of Newton's method } \label{lkant}
In this section,  we study the Newton's method for solving \eqref{eq:ipi}. For study the convergence properties of this method, we assume that  the derivative $F'$ satisfies a weak Lipschitz condition on a region $\Omega$  relaxing  the usual Lipschitz condition. The statement of  the our main result is:
\begin{theorem}\label{th:nt}
Let $\banachh$ be a Hilbert space, $\Omega$ be an open nonempty subset of $\banachh$, $F: \Omega \to \banachh$ be continuous with Fr\'echet derivative $F'$ continuous, $T:\banachh \rightrightarrows  \banachh$ be a set-valued operator and $x^*\in \Omega$. Suppose that  $0 \in F(x^*) +T(x^*)$, $F'(x^*)$ is a positive operator and $\widehat{F'(x^*)}^{-1}$ exists. Let $R>0$  and  $\kappa:=\sup\{t\in [0, R): B(x^*, t)\subset \Omega\}$. Suppose that there exists $f:[0,\; R)\to \mathbb{R}$ twice continuously differentiable such that
  \begin{equation}\label{Hyp:MH}
\|\widehat{F'(x^*)}^{-1}\| \left\|F'(x)-F'(x^*+\tau(x-x^*))\right\| \leq f'\left(\|x-x^*\|\right)-f'\left(\tau\|x-x^*\|\right),
  \end{equation}
  for all $\tau \in [0,1]$, $x\in B(x^*, \kappa)$ and
\begin{itemize}
  \item[{\bf h1)}]  $f(0)=0$ and $f'(0)=-1$;
  \item[{\bf  h2)}]  $f'$ is convex and strictly increasing.
\end{itemize}
  Let  $\nu:=\sup\{t\in [0, R): f'(t)< 0\},$ $\rho:=\sup\{t\in (0, \nu): f(t)/(tf'(t))-1<1\} $  and $r:=\min \left\{\kappa, \,\rho \right\}.$
Then, the sequences with starting point $x_0\in B(x^*, r)/\{x^*\}$ and $t_0=\|x^*-x_0\|$, respectively,
\begin{equation} \label{eq:DNS}
 0\in F(x_k)+F'(x_k)(x_{k+1}-x_k)+T(x_{k+1}), \qquad t_{k+1} =|{t_k}-f(t_k)/f'(t_k)|,\qquad k=0,1,\ldots\,,
\end{equation}
are well defined, $\{t_k\}$ is strictly decreasing, is contained in $(0, r)$ and converges to $0$, $\{x_k\}$ is contained in $B(x^*, r)$ and  converges to the point $x^*$ which is the unique solution of the generalized equation $F(x)+T(x)\ni 0$ in $B(x^*, \bar{\sigma})$, where $\bar{\sigma}=\min \{r, \sigma\}$ and $\sigma:=\sup\{0<t<\kappa: f(t)< 0\}$. Moreover, the sequence $\{t_{k+1}/ t_k^2\}$ is strictly decreasing,
\begin{equation}
    \label{eq:q2}
    \|x^*-x_{k+1}\| \leq \left[t_{k+1}/t_k^2\right]\|x_k-x^*\|^2, \qquad t_{k+1}/t_k^2\leq f''(t_0)/(2 |f'(t_0)|), \qquad k=0,1,\ldots\,.
  \end{equation}
If, additionally,  $f(\rho)/(\rho f'(\rho))-1=1$ and $\rho < \kappa$, then  $r=\rho$ is the optimal convergence radius.
\end{theorem}
\begin{remark}
 Combining inequalities  in \eqref{eq:q2}, we obtain that  $\{x_k\}$ converges  $Q$-quadratically to $\bar{x}$. Moreover, as $\{t_{k+1}/ t_k^2\}$ is strictly decreasing we have  $t_{k+1}/ t_k^2<t_1/t_0^2$, for  $k=0, 1, \ldots.$  Thus, first inequality in \eqref{eq:q2} implies $\|\bar{x}-x_{k+1}\| \leq \left[t_{1}/t_0^2\right]\|x_k-\bar{x}\|^2,$
for $ k=0,1,\ldots\, $. As a consequence,
$$
\|\bar{x}-x_{k}\| \leq t_0\left(t_1/t_0\right)^{2^k-1}, \qquad k=0,1,\ldots\,.
$$
\end{remark}
\begin{remark}\label{def.good}
Since $T$ is monotone maximal, if there exists a constant $c>0$ such that
\begin{equation}\label{eq.gooddef}
\langle F'(x_k)y,y\rangle \geq c\|y\|^2
\end{equation}
for each $y\in \banachh$, then there exists a unique point $x_{k+1}$ such that the first inclusion in \eqref{eq:DNS} holds. The proof of this result can be found in \cite[Lemma~2.2]{Uko1996}. Hence, if for each $k$, there exists a constant $c>0$ such that \eqref{eq.gooddef} holds, then the sequence generated by \eqref{eq:DNS} is well defined.
\end{remark}
{\it From now on, we  assume that the hypotheses of Theorem \ref{th:nt} hold}.
\subsection{Basic results} \label{sec:PMF}
In this section, we establish   some  relationships between the majorant function $f$ and the set-valued mapping  $F+T$. The Proposition~2.5 of \cite{Ferreira2009} state that   the constants $\kappa$,  $\nu$ and $\sigma$ are all positive and $t-f(t)/f'(t)<0,$ for all $t\in (0,\,\nu).$  According to {\bf h2} and  definition of $\nu$, we have  $f'(t)< 0$ for all
$t\in[0, \,\nu)$.  Therefore, the Newton iteration map for $f$ is well defined in
$[0,\, \nu)$, namely, $ n_{f}:[0,\, \nu)\to (-\infty, \, 0]$  is defined by
\begin{equation} \label{eq:def.nf}
n_{f}(t):=t-f(t)/f'(t), \qquad  t\in [0,\, \nu).
\end{equation}
The next proposition was proved in Proposition~2.6 and Proposition~2.7 of \cite{Ferreira2009}.
\begin{proposition}  \label{pr:incr2}
The mapping  $(0,\, \nu) \ni t \mapsto |n_{f}(t)|/t^2$  is  strictly increasing and
$$
|n_{f}(t)|/t^2\leq f''(t)/(2|f'(t)|),
$$
for  all $t\in (0,\, \nu).$ Moreover, the constant $ \rho $ is positive. As a consequence,  $|n_{f}(t)|<t$ for all $ t\in (0, \, \rho)$.
\end{proposition}
Using \eqref{eq:def.nf}, it is easy to see that the sequence $\{t_k \}$ is equivalently defined as
\begin{equation} \label{eq:tknk}
 t_0=\|x^*-x_0\|, \qquad t_{k+1}=|n_{f}(t_k)|, \qquad k=0,1,\ldots\, .
\end{equation}
Next result contain the  main convergence  properties of the above sequence and its prove is similar to  Corollary~2.8 of \cite{Ferreira2009}.
\begin{corollary} \label{cr:kanttk}
The sequence $\{t_k\}$ is well defined, is strictly decreasing and is contained in $(0, \rho)$. Moreover, $\{t_{k+1}/ t_k^2\}$ is strictly decreasing, $\{t_k\}$ converges to $0$ and $t_{k+1}/t_k^2\leq [f''(t_0)/(2|f'(t_0)|)],$ for $k=0,1, \ldots.$
\end{corollary}
In the sequel we will  prove that the partial linearization of $F+T$ has a single-valued inverse, which is Lipschitz  in a  neighborhood of $x^*$. Since  Newton's iteration at a point in this neighborhood  happens to be a zero of the partial linearization of $F+T$ at such a point, it will be first  convenient to study the  {\it linearization error   of  $F$} at a point
in $\Omega$
\begin{equation}\label{eq:def.er}
  E_F(x,y):= F(y)-\left[ F(x)+F'(x)(y-x)\right],\qquad y,\, x\in \Omega.
\end{equation}
In the next result we  bound this error by the linearization error  of the majorant function $f$, namely,
$$
 e_{f}(t,u):=f(u)-\left[f(t)+f'(t)(u-t)\right],\qquad t,\,u \in [0,R).
$$
\begin{lemma}  \label{pr:taylor}
There holds $\|\widehat{F'(x^*)}^{-1}\| \|E_F(x, x^*)\|\leq e_{f}(\|x^*-x\|, 0),$ for all $x\in B(x^*, \kappa)$.
\end{lemma}
\begin{proof}
Since $x^*+(1-u)(x-x^*)\in B(x^*, \kappa)$, for all $0\leq u\leq 1$ and $F$ is continuously differentiable in $\Omega$, thus the definition of $E_F$ and some simple manipulations yield
$$
\|\widehat{F'(x^*)}^{-1}\|\|E_F(x,x^*)\|\leq \int_0 ^1 \|\widehat{F'(x^*)}^{-1}\| \left \| F'(x)-F'(x^*+(1-u)(x-x^*))]\right\|\,\left\|x^*-x\right\| \; du.
$$
Combining last inequality with \eqref{Hyp:MH} with $\tau =1-u$ and then performing the integral obtained using that $f(0)=0$ we obtain that
\begin{eqnarray*}
\|\widehat{F'(x^*)}^{-1}\|\|E_F(x,x^*)\| &\leq& \int_0 ^1 [f'(\|x^*-x\|)-f'((1-u)\|x^*-x\|)]\|x^*-x\| \; du\\
                                         &=& f'(\|x^*-x\|)\|x^*-x\|- f(\|x^*-x\|).
\end{eqnarray*}
Therefore using $h_1$ and the definition of $e_{f}$ the statement follows.
\end{proof}
In the next result  we will present the main relationships between the majorant function $f$ and the operator $F$. The  result is a consequence of Banach's lemma and its statement  is:
\begin{lemma} \label{le:wdns}
Let $x^*\in \banachh$ be such that $\widehat{F'(x^*)}$ is a positive operator and $\widehat{F'(x^*)}^{-1}$ exists. If $\|x-x^*\|\leq \min\{\kappa, \nu\}$, then $\widehat{F'(x)}$ is a positive operator and $\widehat{F'(x)}^{-1}$ exists. Moreover,
$$
\|\widehat{F'(x)}^{-1}\|\leq  \frac{\|\widehat{F'(x^*)}^{-1}\|}{|f'(\|x-x^*\|)|}.
$$
\end{lemma}
\begin{proof}
Firstly note that
\begin{equation}\label{eq.matriz}
\|\widehat{F'(x)}-\widehat{F'(x^*)}\|\leq \frac{1}{2}\|F'(x)-F'(x^*)\| + \frac{1}{2}\|(F'(x)-F'(x^*))^*\|=\|F'(x)-F'(x^*)\|.
\end{equation}
Take $x\in B(x^*, r)$. Since $ r<\nu$ we have $\| x-x^*\|<\nu$. Thus, $f'(\|x-x^*\|)<0$ which, together \eqref{Hyp:MH} and {\bf h1}, taking into account \eqref{eq.matriz}, imply that for all $x\in B(x^*, r)$
  \begin{equation}\label{eq:majcond}
    \|\widehat{F'(x^*)}^{-1}\| \|\widehat{F'(x)}-\widehat{F'(x^*)}\|\leq \|\widehat{F'(x^*)}^{-1}\| \|F'(x)-F'(x^*)\| \leq f'(\|x-x^*\|)-f'(0)<1.
  \end{equation}
	Thus, by Banach's lemma, we conclude that $\widehat{F'(x)}^{-1}$ exists. Moreover by above inequality,
	$$
	\|\widehat{F'(x)}^{-1}\|\leq \frac{\|\widehat{F'(x^*)}^{-1}\|}{1-\|\widehat{F'(x^*)}^{-1}\|\|F'(x)-F'(x^*)\|}\leq \frac{\|\widehat{F'(x^*)}^{-1}\|}{1-(f'(\|x-x^*\|)-f'(0))} =\frac{\|\widehat{F'(x^*)}^{-1}\|}{|f'(\|x-x^*\|)|}.
	$$
The last result follows by noting that $r=\min\{\kappa, \nu\}$. On the other hand, using \eqref{eq:majcond} we have
\begin{equation}\label{eq.selfadj}
\|\widehat{F'(x)}-\widehat{F'(x^*)}\|\leq \frac{1}{\|\widehat{F'(x^*)}^{-1}\|}.
\end{equation}
Take $y\in \banachh$. Then, it follows by above inequality that
$$
\langle (\widehat{F'(x^*)} -\widehat{F'(x)})y,y\rangle \leq \|\widehat{F'(x^*)} -\widehat{F'(x)}\|\|y\|^2\leq \frac{\|y\|^2}{\|\widehat{F'(x^*)}^{-1}\|},
$$
which implies, after of simple manipulations that
$$
\langle \widehat{F'(x^*)}y,y\rangle -\frac{\|y\|^2}{\|\widehat{F'(x^*)}^{-1}\|} \leq \langle \widehat{F'(x)}y,y\rangle.
$$
Since $\widehat{F'(x^*)}$ is a positive operator and $\widehat{F'(x^*)}^{-1}$ exists by assumption, we obtain by Lemma~\ref{eq:adjunt} that
$$
\langle \widehat{F'(x^*)}y,y\rangle \geq \frac{\|y\|^2}{\|\widehat{F'(x^*)}^{-1}\|}.
$$
Therefore, combining the two last inequalities we conclude that $\langle \widehat{F'(x)}y,y\rangle \geq 0$, i.e., $\widehat{F'(x)}$ is a positive operator.
\end{proof}
Lemma~\ref{le:wdns} shows that $\widehat{F'(x)}$ is a positive operator and $\widehat{F'(x)}^{-1}$ exists, thus by Lemma~\ref{eq:adjunt} we have that for any $y\in \banachh$
$$
\langle \widehat{F'(x)}y,y\rangle \geq \frac{\|y\|^2}{\|\widehat{F'(x)}^{-1}\|}.
$$
Note that $\langle \widehat{F'(x)}y,y\rangle=\langle F'(x)y,y\rangle$, thus by the second part of Lemma~\ref{le:wdns} and $h_2$ we conclude that $F'(x)$ satisfies \eqref{eq.gooddef} and consequently, the Newton iteration mapping  is well-defined.  Let us call $N_{F+T}$, the Newton iteration mapping  for $F+T$ in that region, namely, $N_{F+T}:B(x^*, r) \to \banachh$ is defined by
\begin{equation} \label{eq:NFef}
0\in F(x)+F'(x)(N_{F+T}(x)-x)+T(N_{F+T}(x)),\qquad \forall ~x\in  B(x^*, r).
\end{equation}
Therefore, one can apply a \emph{single} Newton iteration on any $x\in B(x^*, r)$ to obtain $N_{F+T}(x)$ which may not belong
to $B(x^*, r)$, or even may not belong to the domain of $F$. Thus, this is enough to guarantee the  well-definedness of only one iteration of Newton's method. To ensure that Newtonian iterations may be repeated indefinitely, we need an additional result.
\begin{lemma} \label{le:cl}
Take $0<t<r$. If $\|x-x^*\|\leq t$ then $\|N_{F+T}(x)-x^*\|\leq [ |n_{f}(t)|/t^2]\,\|x-x^*\|^2.$  As a consequence,
$N_{F+T}(B[x^*, t] ) \subset B[x^*, |n_{f}(t)|].$
Moreover, $N_{F+T}(B(x^*, r))\subset B(x^*, r).$
\end{lemma}
\begin{proof}
Since $0\in F(x^*) + T(x^*)$  we have $ x^*=N_{F+T}(x^*)$. Thus,  the inequality of the lemma is trivial for $x=x^*$. Now, assume that  $0<\|x-x^*\|\leq t$. Let $y=N_{F+T}(x)$. By \eqref{eq:NFef} we have $0\in F(x)+F'(x)(y-x)+T(y)$ for all $x\in  B(x^*, r)$. As $T$ is a maximal monotone, it follows that
$$
\langle F(x)-F(x^*)+F'(x)(x^*-x)+F'(x)(y-x^*), x^*-y \rangle \geq 0
$$
which implies that
\begin{equation}\label{eq.monmax}
\langle F(x)-F(x^*)+F'(x)(x^*-x), x^*-y \rangle \geq \langle F'(x)(x^*-y), x^*-y \rangle.
\end{equation}
Since, by Lemma~\ref{le:wdns}, $\widehat{F'(x)}$ is a positive operator and $\widehat{F'(x)}^{-1}$ exists, we obtain from Lemma~\ref{eq:adjunt} that
 \begin{equation}\label{eq.monmax1}
\frac{\|x^*-y\|^2}{\|\widehat{F'(x)}^{-1}\|}\leq  \langle \widehat{F'(x)}(x^*-y),x^*-y\rangle.
\end{equation}
Note that
$
\langle \widehat{F'(x)}(x^*-y),x^*-y\rangle = \langle F'(x)(x^*-y),x^*-y\rangle,
$
this together with \eqref{eq.monmax1} and \eqref{eq.monmax} yields that
\begin{equation*}\label{eq.monmax21}
\|x^*-y\|^2 \leq \|\widehat{F'(x)}^{-1}\| \langle F'(x)(x^*-y),x^*-y\rangle \leq \|\widehat{F'(x)}^{-1}\| \langle F(x)-F(x^*)+F'(x)(x^*-x), x^*-y\rangle.
\end{equation*}
Hence, after simple manipulations, above inequality becomes
\begin{equation}\label{eq.monmax2}
\|x^*-y\| \leq \|\widehat{F'(x)}^{-1}\| \|F(x)-F(x^*)+F'(x)(x^*-x)\|.
\end{equation}
Using \eqref{eq:def.er}, second part in Lemma~\ref{le:wdns} and Lemma~\ref{pr:taylor} in \eqref{eq.monmax2} we obtain that
\begin{equation}\label{eq.monmax3}
\|x^*-y\| \leq  \frac{\|\widehat{F'(x^*)}^{-1}\|}{|f'(\|x-x^*\|)|}\|E_F(x,x^*)\| \leq \frac{e_{f}(\|x-x^*\|, 0)}{|f'(\|x-x^*\|)|}.
\end{equation}
On the other hand, taking into account that $f(0)=0$, the definitions of $e_f$ and $n_f$ imply that
$$
\frac{e_{f}(\|x-x^*\|, 0)}{|f'(\|x-x^*\|)|}=-n_f (\|x-x^*\|)= |n_f (\|x-x^*\|)|.
$$
As $\|x-x^*\|\leq t$, the first part of Proposition~\ref{pr:incr2} gives $|n_f (\|x-x^*\|)|/\|x-x^*\|^2\leq |n_f(t)|/t^2,$ thus the last inequality becomes
$$
\frac{e_{f}(\|x-x^*\|, 0)}{|f'(\|x-x^*\|)|}\leq |n_f(t)|/t^2 \|x-x^*\|^2.
$$
Hence,  the desired inequality follows by combining \eqref{eq.monmax3} and the latter equation.

For proving second part of the lemma, take $x\in B[x^*, t]$.  Since $\|x-x^*\|\leq t$, first part of the lemma implies that  $\|N_{F+T}(x)-x^*\|\leq |n_{f}(t)|,$ and the first inclusion  follows.  Due to  $ r \leq \rho$, second part of Proposition~\ref{pr:incr2} implies that $|n_{f}(t)|\leq t$.  Thus,   the last inclusion is an immediate consequence of the first one.
\end{proof}
In the next result we  obtain the uniqueness of the solution  in the neighborhood $B[\bar{x}, \sigma].$
\begin{lemma}  \label{l:uniq}
Take $t\in (0,r)$ and suppose that $F'(x^*)$ is a positive operator and $\widehat{F'(x^*)}^{-1}$ exists. If $f(t)<0$, i.e., $0$ is the unique zero of $f$ in $[0,t],$ then $x^*$ is the unique solution of \eqref{eq:ipi} in $B[x^*,t].$ As consequence, $x^*$ is the unique solution of \eqref{eq:ipi}  in $B[x^*, \bar{\sigma}].$
\end{lemma}
\begin{proof}
Assume that $y \in B[x^*,t]$ and $0\in F(y)+T(y).$  Then, as $T$ is a maximal monotone and $0\in F(x^*)+T(x^*)$ we obtain that
\begin{equation*}
\langle F(y)-F(x^*), x^*-y\rangle \geq 0,
\end{equation*}
which implies that
$
\langle F(y)-F(x^*)-F'(x^*)(y-x^*)+F'(x^*)(y-x^*), x^*-y\rangle \geq 0
$
and thus
\begin{equation}\label{eq.uni}
\langle F(y)-F(x^*)-F'(x^*)(y-x^*), x^*-y\rangle \geq \langle F'(x^*)(x^*-y), (x^*-y)\rangle.
\end{equation}
Since $F'(x^*)$ is a positive operator and $\widehat{F'(x^*)}^{-1}$ exists, we can apply Lemma~\ref{eq:adjunt} to obtain that
\begin{equation}\label{eq.uni1}
\langle F'(x^*)(x^*-y), (x^*-y)\rangle =\langle \widehat{F'(x^*)}(x^*-y), (x^*-y)\rangle \geq \frac{\|x^*-y\|^2}{\|\widehat{F'(x^*)}^{-1}\|}.
\end{equation}
On the other hand
$$
F(y)-F(x^*)-F'(x^*)(y-x^*)=\int_0^1  [F'(x^*+t(y-x^*))-F'(x^*)](y-x^*)dt.
$$
Combining above equality with \eqref{eq.uni1} and \eqref{eq.uni}, yields that
$$
\|y-x^*\| \leq \int_0^1 \|\widehat{F'(x^*)}^{-1}\| \|F'(x^*+t(y-x^*))-F'(x^*)\|\|(y-x^*)\|dt.
$$
Using \eqref{Hyp:MH} with $x=x^*+t(y-x^*)$ and $\tau=0$  it is easy to conclude from the last equality that
$$
\|y-x^*\| \leq \int_0^1 [f'(t\|y-x^*\|)-f'(0)]\|y-x^*\|dt= f(\|y-x^*\|)-f(0)-f'(0)\|y-x^*\|.
$$
Taking into account that $f(0)=0$  and $f'(0)=-1$   the latter inequality becomes
$$
f(\|y-x^*\|) \geq 0.
$$
Since $f$ is strictly convex and $f(t)<0$, we will have $f<0$ in $(0,t]$, i.e., $0$ is the unique zero of $f$ in $[0,t]$ and hence, the above inequality implies that $\|y-x^*\|=0$, i.e., $y=x^*$. Thus, $x^*$ is the unique zero of $F+T\ni 0$ in $B[x^*,t]$. The second part follows from the definition of $\sigma$.
\end{proof}
In the next result  we will obtain the  the optimal convergence radius, which has its  proof similar  to the proof of Lemma 2.15 of \cite{Ferreira2009}.
\begin{lemma} \label{pr:best}
If  $f(\rho)/(\rho f'(\rho))-1=1$ and $\rho < \kappa$, then  $r=\rho$ is the optimal convergence radius.
\end{lemma}
\subsection{Proof of {\bf Theorem \ref{th:nt}}} \label{sec:proof}
Firstly, it is easy to see that the inclusion  in \eqref{eq:DNS} together  \eqref{eq:NFef} imply that   the sequence $\{x_k\}$  satisfies
\begin{equation} \label{NFS}
0\in F(x_k)+F'(x_k)(N_{F+T}(x_k)-x_k)+ T(N_{F+T}(x_k)),\qquad k=0,1,\ldots \,.
\end{equation}

\begin{proof}
That $\{t_k\}$ is well defined, is strictly decreasing and is contained in $(0, \rho)$ follows from Corollary \ref{cr:kanttk}. Moreover, from this same corollary, we conclude that $\{t_{k+1}/ t_k^2\}$ is strictly decreasing, $\{t_k\}$ converges to $0$ and $t_{k+1}/t_k^2\leq [f''(t_0)/(2|f'(t_0)|)],$ for $k=0,1, \ldots.$

As $x_0\in B(x^*,r)$, and $r\leq \nu$, we conclude by combining  \eqref{NFS} and   inclusion  $N_{F+T}(B(x^*, r)) \subset B(x^*, r)$ in second part of Lemma~\ref{le:cl} that   $\{x_k\}$  is well defined and remains  in $B(x^*, r)$.  On the other hand, since $0< \|x_k-x^*\|<r\leq \rho,$ for $k=0,1,\ldots,$ we obtain from \eqref{NFS}, Lemma~\ref{le:cl} and second part of Proposition~\ref{pr:incr2} that
\begin{equation}\label{eq:conv}
\|x_{k+1}-x^*\| \leq |n_{f}(\|x_k-x^*\|)|<\|x_k-x^*\|, \qquad k=0,1,\ldots .
\end{equation}
Thence, $\{\|x_k-x^*\|\}$ is strictly decreasing and convergent. Let $b=\lim_{k \to \infty} \|x_k-x^*\|.$ Because $\{\|x_k-x^*\|\}$ rest in $(0,\rho)$ and it is  strictly decreasing we have $0\leq b<\rho.$ Then, by continuity of $n_{f}$ and \eqref{eq:conv} imply $0\leq b=|n_{f}(b)|,$ and from second part of Proposition~\ref{pr:incr2} we have $b=0.$ Therefore, we conclude that  $\{x_k\}$ converges to $x^*$.    Due to $t_0=\|x^*-x_0\|$,  definition  of $\{t_k\}$ in \eqref{eq:tknk} implies that $t_{k+1}=|n_{f}(t_k)|$, hence    \eqref{NFS}  and Lemma~\ref{le:cl}  imply that
 $$
 \|x_{k+1}-x^*\|=\|N_{F+T}(x_k)- x^*\|\leq |n_f(t_k)|,  \qquad k=0,1,\ldots.
 $$
Then, the  first inequality in  \eqref{eq:q2} follows from last inequality,   first part of  Lemma~\ref{le:cl} and    the definition of $\{t_k \}$ in \eqref{eq:tknk}. Finally, the uniqueness follows from  Lemma~\ref{l:uniq} and   the last statement in the theorem follows from   Lemma~\ref{pr:best}.
\end{proof}
\section{Some special cases} \label{apl}
In this section, we will present some  special cases of  Theorem \ref{th:nt}. When    $F\equiv \{0\}$  and    $f'$ satisfies a Lipschitz-type condition,  we will obtain a particular instance of Theorem~\ref{th:nt}, which retrieves  the classical convergence theorem on Newton's method under the Lipschitz condition; see \cite{Rall1974, Traub1979}.  A version of Smale's theorem on Newton's method for analytical functions is obtained in Theorem~\ref{theo:Smale}.
\subsection{Under Lipschitz-type condition}
In this section,  we will present a version of  classical convergence theorem for Newton's method under Lipschitz-type condition  for generalized equations. The classical version for  $F\equiv \{0\}$   have  appeared   in  Rall \cite{Rall1974} and Traub and  Wozniakowski \cite{Traub1979}.
\begin{theorem} \label{th:cc}
Let $\banachh$ be a Hilbert space, $\Omega$ be an open nonempty subset of $\banachh$, $F: \Omega \to \banachh$ be continuous with Fr\'echet derivative $F'$ continuous, $T:\banachh \rightrightarrows  \banachh$ be a set-valued operator and $x^*\in \Omega$. Suppose that  $0 \in F(x^*) +T(x^*)$, $F'(x^*)$ is a positive operator and $\widehat{F'(x^*)}^{-1}$ exists and, there exists a constant $K>0$ such that
\begin{equation} \label{eq:hc}
\|\widehat{F'(x^*)}^{-1}\| \|f'(x)-f'(y)\| \leq K \|x-y\|,\qquad x,\, y\in \Omega.
\end{equation}
Let $r:=\min \left\{\kappa, \,2/(3K) \right\}$, where $\kappa:=\sup\{t>0: B(x^*, t)\subset \Omega\}$. Then, the sequences with starting point $x_0\in B(x^*, r)/\{x^*\}$ and $t_0=\|x^*-x_0\|$, respectively,
\begin{equation} \label{eq:Kant}
 F(x_k)+F'(x_k)(x_{k+1}-x_k)+T(x_{k+1})\ni 0, \qquad t_{k+1} =\left((K/2) \, t_k^2\right)/(1-Kt_k),\qquad k=0,1,\ldots\,,
\end{equation}
are well defined, $\{t_k\}$ is strictly decreasing, is contained in $(0, r)$ and converges to $0$, $\{x_k\}$ is contained in $B(x^*, r)$ and  converges to the point $x^*$ which is the unique solution of $F(x)+T(x)\ni 0$ in $B(x^*, \bar{\sigma})$, where $\bar{\sigma}=\min \{r,  2/K\}$. Moreover, $\{t_{k+1}/ t_k^2\}$ is strictly decreasing, $t_{k+1}/t_k^2<1/[2/K-2\|x^*-x_0\|]$  and
\begin{equation}\label{eq:converratefinal}
 \|x^*-x_{k+1}\| \leq \frac{K}{2} \, \frac{1}{1-Kt_k}\,\|x_k-x^*\|^2  \leq \frac{K}{2}\, \frac{1}{1-K\|x_0-x^*\|}\,\|x_k-x^*\|^2,\qquad  k=0,1, \ldots .
\end{equation}
If, additionally,  $2/(3K)<\kappa$, then $r=2/(3K)$ is the best  possible convergence radius.
\end{theorem}
\begin{proof}
Using condition in \eqref{eq:hc},  we can immediately prove that  $F$, $x^*$ and $f:[0, \kappa)\to \mathbb{R}$, defined by
$
f(t)=Kt^{2}/2-t,
$
satisfy the inequality \eqref{Hyp:MH} and the conditions  {\bf h1} and  {\bf h2} in Theorem \ref{th:nt}. In this case, it is easy to see that  $\rho$ and $\nu$, as defined in Theorem \ref{th:nt}, satisfy
$
\rho= 2/(3K) \leq \nu=1/K
$
and, as a consequence,  $r:=\min \{\kappa,\; 2/(3K)\}$. Moreover, $f(\rho)/(\rho f'(\rho))-1=1$, $f(0)=f(2/K)=0$ and $f(t)<0$ for all $t\in (0,\, 2/K)$. Also,  the sequence $\{t_k\}$  in  Theorem \ref{th:nt} is given by  \eqref{eq:Kant}  and
$$
t_{k+1}/t_k^2=\frac{K}{2} \, \frac{1}{1-Kt_k}<\frac{K}{2}\, \frac{1}{1-K\|x_0-x^*\|},  \qquad k=0,1,\ldots.
$$
Therefore, the result follows  by invoking Theorem~\ref{th:nt}.
\end{proof}
\begin{remark}
The  above result  contain,  as particular instance,  several   theorem on Newton's method; see,  for example, Rall \cite{Rall1974},  Traub and  Wozniakowski \cite{Traub1979} and  Daniel \cite{Daniel1973}.
\end{remark}
\begin{remark}
Since $\|x^*-x_0\|\leq 2/(3K)$, the last inequality in \eqref{eq:converratefinal} implies that $\|x^*-x_{k+1}\|\leq 3K/2\|x^*-x_k\|^2$ for $k=0,1,\ldots.$ Then, we conclude that
$$
\|x^*-x_k\| \leq \frac{2}{3K}\left(\frac{3K}{2}\|x^*-x_0\|\right)^{2^k}, \quad k=0,1,\ldots.
$$
\end{remark}
\subsection{Under Smale's-type condition}
In this section,  we will present a version of  classical convergence theorem for Newton's method under Smale's-type condition for generalized equations. The classical version  has  appeared   in   corollary of Proposition 3 pp.~195 of Smale \cite{Smale1986}, see also Proposition 1 pp.~157 and Remark 1 pp.~158 of  Blum, Cucker,  Shub, and Smale~\cite{BlumSmale1998}; see also \cite{Ferreira2009}.

\begin{theorem} \label{theo:Smale}
Let $\banachh$ be a Hilbert space, $\Omega$ be an open nonempty subset of $\banachh$, $F: \Omega \to \banachh$ be an analytic function, $T:\banachh \rightrightarrows  \banachh$ be a set-valued operator and $x^*\in \Omega$. Suppose that  $0 \in F(x^*) +T(x^*)$, $F'(x^*)$ is a positive operator and $\widehat{F'(x^*)}^{-1}$ exists. Suppose that
\begin{equation} \label{eq:SmaleCond}
   \gamma:= \|\widehat{F'(x^*)}^{-1}\|\sup _{ n > 1 }\left\| \frac  {F^{(n)}(x^*)}{n!}\right\|^{1/(n-1)}<+\infty.
\end{equation}
Let $r=\min \{\kappa,\; (5-\sqrt{17})/(4\gamma)\}$, where $\kappa:=\sup\{t>0 ~:~ B(x^*, t)\subset \Omega\}$. Then, the sequences  with starting point $x_0\in B(x^*, r)/\{x^*\}$ and $t_0=\|x^*-x_0\|$, respectively
$$
0\in F(x_k)+F'(x_k)(x_{k+1}-x_k)+T(x_{k+1}), \qquad  t_{k+1}=\gamma t_k^2/[2(1- \gamma t_k)^2-1],\qquad k=0,1,\ldots\,,
$$
are well defined, $\{t_k\}$ is strictly decreasing, contained in $(0, r)$ and converges to $0$, and $\{x_k\}$ is contained in $B(x^*,r)$ and  converges to the point $x^*$ which is the unique solution of $F(x)+T(x)\ni 0$ in $B(x^*, \bar{\sigma})$, where $\bar{\sigma}=\min \{r, 1/(2\gamma)\}$.  Moreover, $\{t_{k+1}/ t_k^2\}$ is strictly decreasing, $t_{k+1}/t_k^2<\gamma/[2(1-\gamma \|x_0-x^*\|)^2-1]$,   for  $  k=0,1,\ldots $ and
$$
    \|x_{k+1}-x^*\| \leq  \frac{ \gamma}{2(1- \gamma t_k)^2-1}\;\|x_k-x^*\|^2\leq \frac{ \gamma}{2(1- \gamma \|x_0-x^*\|)^2-1}\;\|x_k-x^*\|^2, \qquad k=0,1, \ldots.
$$
If, additionally,  $(5-\sqrt{17})/(4 \gamma)<\kappa$, then $r=(5-\sqrt{17})/(4\gamma)$ is the best  possible convergence radius.
\end{theorem}
Before proving above theorem we need of two results. The next results gives a condition  that is easier to check than condition
\eqref{Hyp:MH}, whenever  the functions under consideration  are twice continuously differentiable, and its proof follows the same path of Lemma~21 of \cite{FerreiraGoncalvesOliveira2011}.
\begin{lemma}\label{lem.cond1}
Let $\Omega \subset \banachh$ be an open set, and let $F:{\Omega}\to \banachh$ be an analytic function. Suppose that $x^* \in \Omega$ and $B(x^*, 1/ \gamma)\subset \Omega,$ where $\gamma$ is defined in \eqref{eq:SmaleCond}. Then for all $x\in B(x^*, 1/  \gamma),$ it holds that
$
\|F''(x)\|\leq 2  \gamma/(1-  \gamma\|x-x^*\|)^3.
$
\end{lemma}
The next result  gives a relationship  between  the second derivatives $F''$ and $ f''$,  which allow us to show that  $F$ and $f$ satisfy \eqref{Hyp:MH}, and its proof  is similar to Lemma~22 of \cite{FerreiraGoncalvesOliveira2011}.
\begin{lemma} \label{lc}
Let $\banachh$ be a Hilbert space, $\Omega\subseteq \banachh$ be an open set,
  $F:{\Omega}\to \banachh$  be twice continuously differentiable. Let $x^* \in \Omega$, $R>0$  and  $\kappa=\sup\{t\in [0, R): B(x^*, t)\subset \Omega\}$. Let \mbox{$f:[0,R)\to \mathbb {R}$} be twice continuously differentiable such that $ \|\widehat{F'(x^*)}^{-1}\|\|F''(x)\|\leqslant f''(\|x-x^*\|),$
for all $x\in B(x^*, \kappa)$, then $F$ and $f$ satisfy \eqref{Hyp:MH}.
\end{lemma}

\noindent
{\bf [Proof of Theorem \ref{theo:Smale}]}.
Consider $f:[0, 1/ \gamma) \to \mathbb{R}$ defined by $f(t)=t/(1- \gamma t)-2t$. Note that $f$ is  analytic and
$f(0)=0$,   $f'(t)=1/(1- \gamma t)^2-2$, $f'(0)=-1$, $f''(t)=2 \gamma/(1-\gamma t)^3$.
It follows from the last  equalities  that $f$ satisfies {\bf h1}  and  {\bf h2}.  Combining  Lemma~\ref{lc}  with  Lemma~\ref{lem.cond1}, we conclude  that $F$  and $f$ satisfy  \eqref{Hyp:MH}.  The constants,  $\nu$,  $\rho$ and $r$, as defined in Theorem \ref{th:nt}, satisfy
$$\rho= \frac{5-\sqrt{17}}{4 \gamma}< \nu=\frac{\sqrt{2}-1}{\sqrt{2} \gamma}< \frac{1}{\ \gamma},\qquad r=\min \left\{\kappa,\; \frac{5-\sqrt{17}}{4 \gamma}\right\}.
$$
Moreover, $f(\rho)/(\rho f'(\rho))-1=1$ and $f(0)=f(1/(2 \gamma))=0$ and $f(t)<0$ for $t\in (0,\, 1/(2\gamma))$. Also, $\{t_k\}$  satisfy
$$
t_{k+1}/t_k^2=\frac{ \gamma}{2(1- \gamma t_k)^2-1}<\frac{ \gamma}{2(1- \gamma \|x_0-x^*\|)^2-1} ,  \qquad k=0,1,\ldots.
$$
Therefore, the result follows by applying the Theorem~\ref{th:nt}.
\qed

\end{document}